\documentclass[12pt,oneside,english]{amsart}
\usepackage[T1]{fontenc}
\usepackage[latin9]{inputenc}
\usepackage{geometry}
\usepackage{color}
\usepackage{amstext}
\usepackage{amsthm}
\usepackage{amssymb}
\usepackage{stmaryrd}

\makeatletter
\numberwithin{equation}{section}
\numberwithin{figure}{section}
\theoremstyle{plain}
\newtheorem{thm}{\protect\theoremname}
  \theoremstyle{remark}
  \newtheorem{rem}[thm]{\protect\remarkname}
\newenvironment{lyxlist}[1]
{\begin{list}{}
{\settowidth{\labelwidth}{#1}
 \setlength{\leftmargin}{\labelwidth}
 \addtolength{\leftmargin}{\labelsep}
 }}
{\end{list}}
  \theoremstyle{plain}
  \newtheorem{lem}[thm]{\protect\lemmaname}
  \theoremstyle{definition}
  \newtheorem{example}[thm]{\protect\examplename}
  \theoremstyle{definition}
  \newtheorem{defn}[thm]{\protect\definitionname}
  \theoremstyle{plain}
  \newtheorem{prop}[thm]{\protect\propositionname}
  \theoremstyle{plain}
  \newtheorem{cor}[thm]{\protect\corollaryname}

\usepackage{amsmath}

\makeatother

\usepackage{babel}
  \providecommand{\corollaryname}{Corollary}
  \providecommand{\definitionname}{Definition}
  \providecommand{\examplename}{Example}
  \providecommand{\lemmaname}{Lemma}
  \providecommand{\propositionname}{Proposition}
  \providecommand{\remarkname}{Remark}
\providecommand{\theoremname}{Theorem}

\begin{document}

\title{Grothendieck rings for Lie superalgebras and the Duflo\textendash Serganova
functor }

\author{Crystal Hoyt and Shifra Reif
}
\thanks{The first author was partially supported by BSF Grant 2012227. The
second author was partially supported by ORT Braude College's Research
Authority.}
\begin{abstract}
We show that the Duflo\textendash Serganova functor on the category
of finite-dimensional modules over a finite-dimensional contragredient
Lie superalgebra induces a ring homomorphism on a natural quotient
of the Grothendieck ring, which is isomorphic to the ring of supercharacters.
We realize this homomorphism as a certain evaluation of functions
related to the supersymmetry property. We use this realization to
describe the kernel and image of the homomorphism induced by the Duflo\textendash Serganova
functor. 
\end{abstract}

\maketitle

\section{Introduction}

The Duflo\textendash Serganova functor was originally introduced in
\cite{DS} together with associated varieties of modules over Lie
superalgebras. On the category of finite-dimensional modules, the
Duflo\textendash Serganova functor is a tensor functor which preserves
the superdimension. This functor was used by Serganova in \cite{S11}
to prove the conjecture of Kac and Wakimoto that the superdimension
of a finite-dimensional module is zero if and only if the atypicality
of the module is maximal. The Duflo\textendash Serganova functor was
also used to give an additional proof for the superdimension formula
of $GL\left(m|n\right)$-modules in \cite{HW1}, and has been applied
to study Deligne categories in \cite{CH,EHS,H,HW2}. 

Given an odd element $x$ in a Lie superalgebra $\mathfrak{g}$ satisfying
$\left[x,x\right]=0$, we have that $x^{2}=0$ in the universal enveloping
algebra of $\mathfrak{g}$, and so for every $\mathfrak{g}$-module
$M$, we can define the cohomology 
\[
M_{x}:=\mathrm{Ker}_{M}x/xM.
\]
In fact, $M_{x}$ is a module for the Lie superalgebra 
\[
\mathfrak{g}_{x}:=\mathrm{Ker\ }\mathrm{ad}_{x}/\mathrm{Im}\ \mathrm{ad}_{x},
\]
which is a Lie superalgebra of smaller rank than $\mathfrak{g}$.
For example, if $\mathfrak{\mathfrak{g}=gl}\left(m|n\right)$ and
$x$ is a root vector, then $\text{\ensuremath{\mathfrak{g}}}_{x}=\mathfrak{gl}\left(m-1|n-1\right)$.
Duflo and Serganova defined the functor $DS_{x}:M\mapsto M_{x}$ from
the category of $\mathfrak{g}$-modules to the category of $\mathfrak{g}_{x}$-modules
\cite{DS}, which we refer to as the Duflo\textendash Serganova functor. 

One of the difficulties that arises in using the Duflo\textendash Serganova
functor is that it is not exact. It is therefore surprising that it
induces a ring homomorphism $ds_{x}$ on a natural quotient of the
Grothendieck ring of the category of finite-dimensional $\mathfrak{g}$-modules.
This quotient is defined by identifying the equivalence class of a
module $\left[M\right]$ with $-\left[\Pi\left(M\right)\right]$,
where $\Pi$ is the shift of parity functor. We refer to this quotient
as the \emph{supercharacter ring} of $\mathfrak{g}$ and show that
the homomorphism $ds_{x}$ is indeed well defined. 

Sergeev and Veselov described the supercharacter ring as a ring of
functions admitting a certain supersymmetry condition in \cite{SV}.
In this paper, we realize the homomorphism $ds_{x}$ in terms of evaluation
of functions related to the supersymmetry condition. For example,
the supercharacter ring of the Lie supergroup $GL\left(m|n\right)$
corresponding to the Lie superalgebra $\mathfrak{gl}\left(m|n\right)$
is isomorphic to the ring of doubly symmetric Laurent polynomials
in $x_{1},\ldots,x_{m},y_{1},\ldots,y_{n}$ for which the evaluation
$x_{1}=y_{1}=t$ is independent of $t$. If $x$ is a root vector
for the root $\varepsilon_{i}-\delta_{j}$ of $\mathfrak{gl}\left(m|n\right)$,
then the homomorphism $ds_{x}$ is given by the evaluation $x_{i}=y_{j}=t$,
which is independent of the variable $t$ after evaluation, by the
supersymmetry property.

We use this realization to describe the kernel of the homomorphism
$ds_{x}$ when $x$ is a root vector. In particular, we show that
if $\mathfrak{g}$ is a Lie superalgebra of Type I, the supercharacters
of Kac modules form a basis for the kernel. When $\mathfrak{g}$ is
a Lie superalgebra of Type II, there are no Kac modules; however,
we show that the kernel has a basis consisting of expressions similar
to the supercharacters of Kac modules. These are the same expressions
that were used by Gruson and Serganova to define Kazhdan\textendash Lusztig
polynomials for the orthosymplectic Lie superalgebras \cite{GS}.

We also describe the image of $ds_{x}$. In particular, for $\mathfrak{g}=\mathfrak{sl}\left(m|n\right)$,
$m\neq n$, and $\mathfrak{osp}\left(m|2n\right)$, we show that the
image is the supercharacter ring of $G_{x}$, where $G_{x}$ is the
Lie supergroup corresponding to the Lie superalgebra $\mathfrak{g}_{x}$.
Moreover, we prove that the homomorphism induced by the Duflo\textendash Serganova
functor from the category of finite-dimensional $G$-modules to the
category of finite-dimensional $G_{x}$-modules is surjective. For
the exceptional Lie superalgebras, we explicitly describe the image
using a set of generators. 

\subsection*{Acknowledgment.}

The authors are thankful to Maria Gorelik, Rachel Karpman, Ivan Penkov
and Vera Serganova for fruitful conversations.

\section{Preliminaries}

\subsection{Lie superalgebras}

Lie superalgebras are a natural generalization of Lie algebras which
first appeared in mathematical physics. In this paper, we study the
finite-dimensional contragredient Lie superalgebras $\mathfrak{g}=\mathfrak{g}_{\bar{0}}\oplus\mathfrak{g}_{\bar{1}}$
with indecomposable Cartan matrix. These are the Lie superalgebras
$\mathfrak{sl}\left(m|n\right)\ m\neq n$, $\mathfrak{gl}\left(n|n\right)$,
$\mathfrak{osp}\left(m|2n\right)$, $D\left(2,1,\alpha\right)$, $F(4)$
or G(3). We also consider the case when $\mathfrak{g=\mathfrak{gl}}\left(m|n\right)$
is the general linear Lie superalgebra. These Lie superalgebras resemble
reductive Lie algebras in their structure theory; in particular, they
are defined by a Cartan matrix and they possess an even supersymmetric
invariant bilinear form $\left(\cdot,\cdot\right)$ which has kernel
equal to the center of $\mathfrak{g}$.

Fix a Cartan subalgebra $\mathfrak{h}\subset\mathfrak{g}_{\bar{0}}\subset\mathfrak{g}$,
and consider the corresponding root space decomposition 
\[
\mathfrak{g}=\mathfrak{h}\oplus\bigoplus_{\alpha\in\Delta}\mathfrak{g}_{\alpha}.
\]
Then the set of roots $\Delta\subset\mathfrak{h}^{*}$ splits $\Delta=\Delta_{\bar{0}}\sqcup\Delta_{\bar{1}}$
into even roots $\Delta_{\bar{0}}$ and odd roots $\Delta_{\bar{1}}$.
A choice of positive roots $\Delta^{+}=\Delta_{\bar{0}}^{+}\sqcup\Delta_{\bar{1}}^{+}$
determines a triangular decomposition of $\mathfrak{g}$ given by
$\mathfrak{g}=\mathfrak{n}^{+}\oplus\mathfrak{h}\oplus\mathfrak{n}^{-}$,
where $\mathfrak{n}^{\pm}=\oplus_{\alpha\in\Delta^{\pm}}\mathfrak{g}_{\alpha}$.
Let $\rho_{\bar{0}}=\frac{1}{2}\sum_{\alpha\in\Delta_{\bar{0}}^{+}}\alpha$
, $\rho_{\bar{1}}=\frac{1}{2}\sum_{\alpha\in\Delta_{\bar{1}}^{+}}\alpha$
and $\rho=\rho_{\bar{0}}-\rho_{\bar{1}}$. The Weyl group $W$ of
$\mathfrak{g}$ is by definition the Weyl group of $\mathfrak{g}_{\bar{0}}$.
The sign map $sgn:W\rightarrow\left\{ \pm1\right\} $ is defined by
$w\mapsto\left(-1\right)^{l\left(w\right)}$, where $l\left(w\right)$
denotes the length of $w$ as a product of simple reflections with
respect to a set of simple roots for $\mathfrak{g}_{\bar{0}}$.

The space $\mathfrak{h}^{*}$ inherits an even supersymmetric bilinear
form $\left(\cdot,\cdot\right)$. A root $\beta\in\Delta_{\bar{1}}$
is called isotropic if $\left(\beta,\beta\right)=0$. Two roots $\alpha,\beta\in\Delta$
are called orthogonal if $\left(\alpha,\beta\right)=0$. The maximal
number of linearly independent mutually orthogonal isotropic roots
is called the defect of $\mathfrak{g}$. We denote by $\Delta_{iso}:=\left\{ \beta\in\Delta_{\bar{1}}\mid\left(\beta,\beta\right)=0\right\} $
the set of all isotropic roots, by $\Delta_{iso}^{+}=\Delta_{iso}\cap\Delta^{+}$
the set of positive isotropic roots and we let $\rho_{iso}:=\frac{1}{2}\sum_{\alpha\in\Delta_{iso}^{+}}\alpha$.
We define
\begin{equation}
\mathcal{S}_{\mathfrak{g}}=\left\{ B\subset\Delta_{iso}\mid B=\left\{ \beta_{1},\ldots,\beta_{k}\mid\left(\beta_{i},\beta_{j}\right)=0,\ \beta_{i}\ne\pm\beta_{j}\right\} \right\} \label{eq:iso}
\end{equation}
 to be the set of subsets of linearly independent mutually orthogonal
isotropic roots.

The space $\mathfrak{h}^{*}$ has a natural basis $\varepsilon_{1},\ldots,\varepsilon_{m},\delta_{1},\ldots,\delta_{n}$,
which for $\mathfrak{gl}\left(m|n\right)$ and $\mathfrak{osp}\left(m|2n\right)$
satisfies $(\varepsilon_{i},\varepsilon_{j})=\delta_{ij}=-\left(\delta_{i},\delta_{j}\right)$
and $\left(\varepsilon_{i},\delta_{j}\right)=0$. The roots of $\mathfrak{g}$
have a nice presentation in this basis (see \cite{CW,Mu} for more
details). Let $Q_{\mathfrak{g}}=\mathrm{span}_{\mathbb{Z}}\Delta$
be the root lattice of $\mathfrak{g}$, and let $Q_{\mathfrak{g}}^{+}=\mathrm{span}_{\mathbb{Z}}\Delta^{+}$.
The parity function $p:\Delta\rightarrow\mathbb{Z}_{2}$ extends uniquely
to a linear function $p:Q_{\mathfrak{g}}\rightarrow\mathbb{Z}_{2}$.
The root lattice $Q_{\mathfrak{g}}$ is contained in the integral
weight lattice $P_{\bar{0}}$ for $\mathfrak{g}_{\bar{0}}$, where
\[
P_{\bar{0}}=\left\{ \lambda\in\mathfrak{h}^{*}\mid\frac{2\left(\lambda,\alpha\right)}{\left(\alpha,\alpha\right)}\in\mathbb{Z}\ \forall\alpha\in\Delta_{\bar{0}}\right\} .
\]
 The set of dominant integral weights 
\[
P_{\bar{0}}^{+}=\left\{ \lambda\in P_{\bar{0}}\mid\frac{2\left(\lambda,\alpha\right)}{\left(\alpha,\alpha\right)}\geq0\ \forall\alpha\in\Delta_{\bar{0}}\right\} 
\]
 is the set of highest weights of finite-dimensional simple $\mathfrak{\mathfrak{g}}_{\bar{0}}$-modules.

The category of finite-dimensional modules $\mathcal{F}_{\mathfrak{g}}$
over a Lie superalgebra $\mathfrak{g}$ is not semisimple, that is,
there exist indecomposable modules which are not irreducible. For
example, a Lie superalgebra $\mathfrak{g}$ of Type I has a decomposition
$\mathfrak{g}=\mathfrak{g}_{-1}\oplus\mathfrak{g}_{\bar{0}}\oplus\mathfrak{g}_{+1}$,
so one can define the Kac module of highest weight $\lambda\in P_{\bar{0}}$
as 
\[
K\left(\lambda\right)=\mathrm{\mathrm{Ind}_{\mathfrak{g}_{\bar{0}}\oplus\mathfrak{g}_{+1}}^{\mathfrak{g}}L_{\bar{0}}\left(\lambda\right)},
\]
where $L_{\bar{0}}\left(\lambda\right)$ is the finite-dimensional
simple $\mathfrak{\mathfrak{g}}_{\bar{0}}$-module of highest weight
$\lambda$ and $\mathfrak{g}_{+1}$ acts trivially on $L_{\bar{0}}\left(\lambda\right)$.
Then $K\left(\lambda\right)$ is a finite-dimensional, indecomposable
$\mathfrak{g}$-module with a unique simple quotient $L\left(\lambda\right)$,
where $\lambda$ is the highest weight with respect to the distinguished
choice of simple roots, and $K\left(\lambda\right)$ is simple (i.e.
$K\left(\lambda\right)=L\left(\lambda\right)$) if and only if $\lambda$
is a typical weight: $\left(\lambda+\rho,\beta\right)\neq0$ for all
$\beta\in\Delta_{iso}$ (see, for example, \cite[Chapter 2]{CW} for
more details).

If $G_{\bar{0}}$ is a simply-connected connected Lie group corresponding
to the Lie algebra $\mathfrak{g}_{\bar{0}}$ \cite{S14}, and $\mathcal{F}_{G}$
is the full subcategory of $\mathcal{F}_{\mathfrak{g}}$ consisting
of all finite-dimensional $G_{\bar{0}}$-integrable modules, then
$\mathcal{F}_{G}$ is equivalent to the category of finite-dimensional
modules over the corresponding algebraic supergroup $G$ \cite{S14}. 

\subsection{Supercharacter rings of Lie superalgebras}

The character theory of Lie superalgebras is a rich area of research
which has led to interesting applications in number theory (see \cite{KW1,KW2}).
For a finite-dimensional $\mathfrak{g}$-module $M$, with weight
decomposition $M=\oplus_{\mu\in\mathfrak{h}^{*}}M^{\mu}$ and weight
spaces $M^{\mu}=M_{\bar{0}}^{\mu}\oplus M_{\bar{1}}^{\mu}$ , the
supercharacter of $M$ is defined to be
\[
\operatorname{sch}M=\sum_{\mu\in\mathfrak{h}^{*}}\left(\dim\ M_{\bar{0}}^{\mu}-\dim\ M_{\bar{1}}^{\mu}\right)e^{\mu},
\]
while the character of $M$ is given by $\operatorname{ch}M=\sum\left(\dim\ M_{\bar{0}}^{\mu}+\dim\ M_{\bar{1}}^{\mu}\right)e^{\mu}$.
A finite-dimensional simple $\mathfrak{g}$-module is determined by
its supercharacter, as well as by its character \cite[Prop. 4.2]{SV}. 

The supercharacter ring $\mathcal{J}_{\mathfrak{g}}$ of a Lie superalgebra
$\mathfrak{g}$ is defined to be the image of the map
\[
\operatorname{sch}:\mathcal{F}_{\mathfrak{g}}\rightarrow\mathbb{Z}\left[P_{\bar{0}}\right]^{W},
\]
 where $\mathbb{Z}\left[P_{\bar{0}}\right]:=\mathbb{\mathbb{Z}}\left\{ e^{\mu}\mid\mu\in P_{\bar{0}}\right\} $.
For an element $f\in\mathcal{J}_{\mathfrak{g}}$, with $f=\sum_{\mu\in P_{\bar{0}}}c_{\mu}e^{\mu}$,
we call the set $\operatorname{Supp}f=\left\{ \mu\in P_{\bar{0}}\mid c_{\mu}\neq0\right\} $
the support of $f$. 

For a fixed choice of positive roots $\Delta^{+}=\Delta_{\bar{0}}^{+}\sqcup\Delta_{\bar{1}}^{+}$,
we denote the super Weyl denominator by $R=\frac{R_{\bar{0}}}{R_{\bar{1}}}$
where $R_{\bar{0}}=\prod_{\alpha\in\Delta_{\bar{0}}^{+}}\left(1-e^{-\alpha}\right)$
and $R_{\bar{1}}=\prod_{\alpha\in\Delta_{\bar{1}}^{+}}\left(1-e^{-\alpha}\right)$.
Note that the supercharacter of the Kac module equals
\[
\operatorname{sch}K\left(\lambda\right)=e^{-\rho}R^{-1}\cdot\mbox{ch }L_{\bar{0}}\left(\lambda\right),
\]
 where $\Delta^{+}=\Delta_{\bar{0}}^{+}\sqcup\Delta_{\bar{1}}^{+}$
is the distinguished choice of simple roots.

The Grothendieck group of the category $\mathcal{F}_{\mathfrak{g}}$
is defined by taking the free abelian group generated by the elements
$\left[M\right]$ which represent each isomorphism class of finite-dimensional
$\mathfrak{g}$-modules, and modding out by the relations $\left[M_{1}\right]-\left[M_{2}\right]+\left[M_{3}\right]$
for all exact sequences $0\rightarrow M_{1}\rightarrow M_{2}\rightarrow M_{3}\rightarrow0$.
Since $\mathcal{F}_{\mathfrak{g}}$ is closed under tensor products,
the Grothendieck group inherits a natural ring structure. 

The Grothendieck ring of $\mathcal{F}_{\mathfrak{g}}$ has a natural
quotient described as follows. Let $\Pi$ denote the parity reversing
functor from $\mathcal{F}_{\mathfrak{g}}$ to itself, and let $\mathcal{K}_{\mathfrak{g}}$
denote the quotient of the Grothendieck ring of $\mathcal{F}_{\mathfrak{g}}$
by the ideal $\left\langle \left[\Pi\left(M\right)\right]+\left[M\right]\mid M\text{ is a \ensuremath{\mathfrak{g}}-module}\right\rangle $.
The map $\operatorname{sch}:\mathcal{K}_{\mathfrak{g}}\rightarrow\mathbb{Z}\left[P_{\bar{0}}\right]^{W_{\bar{0}}}$
given on generators by $\left[M\right]\mapsto\operatorname{sch}M$
is injective \cite[Proposition 4.4]{SV}, and its image is the supercharacter
ring $\mathcal{J}_{\mathfrak{g}}$ of $\mathfrak{g}$. 
\begin{rem}
\label{rem:ring identification}In this paper, we identify the rings
$\mathcal{K}_{\mathfrak{g}}$ and $\mathcal{J}_{\mathfrak{g}}$ under
this isomorphism, and use the notation $\mathcal{J}_{\mathfrak{g}}$
to denote this ring. Given a module $M\in\mathcal{F}_{\mathfrak{g}}$,
we write $\left[M\right]$ for its image in $\mathcal{J}_{\mathfrak{g}}$.
\end{rem}
Sergeev and Veselov gave an explicit description of supercharacter
rings for basic classical Lie superalgebras in \cite{SV} as follows.
The supercharacter ring of $\mathfrak{g}$ is isomorphic to the the
space of supersymmetric exponential functions 
\begin{equation}
\mathcal{J}_{\mathfrak{g}}=\left\{ f\in\mathbb{Z}\left[P_{\bar{0}}\right]^{W}\,\mid\,D_{\beta}f\mbox{ is in the ideal generated by }\left(e^{\beta}-1\right)\mbox{ for any }\beta\in\Delta_{iso}\right\} \label{eq:SV character ring}
\end{equation}
where $D_{\beta}\left(e^{\lambda}\right)=\left(\lambda,\beta\right)e^{\beta}$.
Sergeev and Veselov also described the supercharacter ring $\mathcal{J}_{G}\subset\mathcal{J}_{\mathfrak{g}}$
for the Lie supergroup $G$ corresponding to the Lie superalgebra
$\mathfrak{g}$ as a ring of Laurent polynomials subject to some additional
conditions \cite[Section 7]{SV}. Recall the basis $\varepsilon_{1},\ldots,\varepsilon_{m},\delta_{1},\ldots,\delta_{n}$
of $\mathfrak{h}^{*}$ , and define: $x_{i}:=e^{\varepsilon_{i}}$,
$y_{j}:=e^{\delta_{j}}$, $u_{i}=x_{i}+x_{i}^{-1}$, $v_{j}=y_{j}+y_{j}^{-1}$.
\begin{lyxlist}{00.00.0000}
\item [{\emph{$GL(m|n)$:}}] \emph{The supercharacter ring of $GL(m|n)$
is 
\begin{equation}
\mathcal{J}_{G}=\left\{ f\in\mathbb{Z}\left[x_{1}^{\pm1},\ldots,x_{m}^{\pm1},y_{1}^{\pm1},\ldots,y_{n}^{\pm1}\right]^{S_{m}\times S_{n}}\,\mid\,y_{j}\frac{\partial f}{\partial y_{j}}+x_{i}\frac{\partial f}{\partial x_{i}}\in\left\langle y_{j}-x_{i}\right\rangle \right\} .\label{eq:GL K_G}
\end{equation}
}
\item [{\emph{$SL(m|n)$,}}] \emph{$m\neq n$: The supercharacter ring
of $SL(m|n)$ for $m\neq n$ is the quotient of (\ref{eq:GL K_G})
by the ideal $\left\langle x_{1}\ldots x_{m}-y_{1}\ldots y_{n}\right\rangle $.}
\item [{\emph{$B\left(m|n\right)$:}}] \emph{The supercharacter ring of
$OSP\left(2m+1|2n\right)$ is}
\end{lyxlist}
\emph{
\[
\mathcal{J}_{G}=\left\{ f\in\mathbb{Z}\left[u_{1},\ldots,u_{m},v_{1},\ldots,v_{n}\right]^{S_{m}\times S_{n}}\,\mid\,u_{i}\frac{\partial f}{\partial u_{i}}+v_{j}\frac{\partial f}{\partial v_{j}}\in\left\langle u_{i}-v_{j}\right\rangle \right\} .
\]
}
\begin{lyxlist}{00.00.0000}
\item [{$C\left(n+1\right)$:}] The supercharacter ring of $OSP\left(2|2n\right)$
is\emph{
\[
\mathcal{J}_{G}=\left\{ f\in\mathbb{Z}\left[u_{1},v_{1},\ldots,v_{n}\right]^{S_{m}}\,\mid\,u_{1}\frac{\partial f}{\partial u_{1}}+v_{j}\frac{\partial f}{\partial v_{j}}\in\left\langle u_{1}-v_{j}\right\rangle \right\} .
\]
}
\item [{\emph{$D(m|n)$,}}] \emph{$m\ge2$: The supercharacter ring of
$OSP(2m|2n)$ for $m\geq2$ is
\[
\mathcal{J}_{G}=\left\{ f\in\mathbb{Z}\left[u_{1},\ldots,u_{m},v_{1},\ldots,v_{n}\right]^{S_{m}\times S_{n}}\,\mid\,u_{i}\frac{\partial f}{\partial u_{i}}+v_{j}\frac{\partial f}{\partial v_{j}}\in\left\langle u_{i}-v_{j}\right\rangle \right\} .
\]
}
\end{lyxlist}
\begin{rem}
Note that $f\in\mathcal{J}_{GL\left(m|n\right)}$ if and only if it
is supersymmetric in $x_{1},\ldots,x_{m},y_{1},\ldots y_{n}$, that
is, if it is invariant under permutation of $x_{1},\ldots,x_{m}$
and of $y_{1},\ldots,y_{n}$, and if the substitution $x_{1}=y_{1}=t$
made in $f$ is independent of $t$ (see for example \cite[Section 12]{Mu}).
\end{rem}

\subsection{The Duflo\textendash Serganova functor\label{subsec:The-Duflo-Serganova-functors}}

The idea behind the Duflo\textendash Serganova functor is simple and
natural. For any odd element $x\in\mathfrak{g}_{\bar{1}}$ of a finite-dimensional
contragredient Lie superalgebra $\mathfrak{g}$ which satisfies $\left[x,x\right]=0$,
we have that $x^{2}=0$ in the universal enveloping algebra of $\mathfrak{g}$,
and so for any finite-dimensional $\mathfrak{g}$-module $M$ we can
define the cohomology
\begin{equation}
M_{x}:=\mathrm{Ker}_{M}x/xM.\label{eq: Mx}
\end{equation}
 Then $M_{x}$ is in fact a module over the Lie superalgebra 
\[
\mathfrak{g}_{x}:=\mathfrak{g}^{x}/[x,\mathfrak{g}],
\]
 where $\mathfrak{g}^{x}=\left\{ a\in\mathfrak{g}\mid\left[x,a\right]=0\right\} $
is the centralizer of $x$ in $\mathfrak{g}$ \cite[Lemma 6.2]{DS}.
The Duflo\textendash Serganova functor $DS_{x}:\ \mathcal{F}_{\mathfrak{g}}\rightarrow\mathcal{F}_{\mathfrak{g}_{x}}$
is defined from the category of finite-dimensional $\mathfrak{g}$-modules
to the category of finite-dimensional $\mathfrak{g}_{x}$-modules
by sending $M\mapsto M_{x}$. 

The Duflo\textendash Serganova functor is a cohomology functor and
hence is a symmetric monoidal tensor functor, that is, for $\mathfrak{g}$-modules
$M,\ N$ one has a natural isomorphism $M_{x}\otimes N_{x}\rightarrow\left(M\otimes N\right)_{x}$
\cite{S11}. Moreover, the Duflo\textendash Serganova functor commutes
with direct sums, however it is not exact.

Let $X_{\mathfrak{g}}=\left\{ x\in\mathfrak{g}_{\bar{1}}:\left[x,x\right]=0\right\} $,
and let $\mathcal{S_{\mathfrak{g}}}$ be the set of subsets of mutually
orthogonal isotropic roots (see (\ref{eq:iso})). Then the $G_{\bar{0}}$-orbits
of $X_{\mathfrak{g}}$ are in one-to-one correspondence with the $W$-orbits
of $\mathcal{S_{\mathfrak{g}}}$ via the correspondence
\begin{equation}
B=\left\{ \beta_{1},...,\beta_{k}\right\} \mapsto x=x_{\beta_{1}}+\cdots+x_{\beta_{k}}\in X_{\mathfrak{g}},\label{eq:bijection for x and B}
\end{equation}
 where each $x_{\beta_{i}}\in\mathfrak{g_{\beta_{i}}}$ is chosen
to be nonzero \cite[Theorem 4.2]{DS}.

The Lie superalgebra $\mathfrak{g}_{x}$ can be naturally embedded
into $\mathfrak{g}^{x}\subset\mathfrak{g}$, in such a way that $\mathfrak{h}_{x}=\mathfrak{h}\cap\mathfrak{g}_{x}$
is a Cartan subalgebra of $\mathfrak{g}_{x}$ and the root spaces
of $\mathfrak{g}_{x}$ are root spaces of $\mathfrak{g}$ \cite[Lemma 6.3]{DS}.
More explicitly, Duflo and Serganova proved the following.

\emph{If $B=\left\{ \beta_{1},\ldots,\beta_{k}\right\} \in\mathcal{S}$
and $x=x_{\beta_{1}}+\cdots+x_{\beta_{k}}$ for some nonzero $x_{\beta_{i}}\in\mathfrak{g_{\beta_{i}}}$,
then $\mathfrak{g}^{x}\subset\mathfrak{g}$ can be decomposed into
a semidirect sum $\mathfrak{g}^{x}=\left[x,\mathfrak{g}\right]\subsetplus\mathfrak{g}_{x}$,
where $\mathfrak{g}_{x}=\mathfrak{h}_{x}\oplus\left(\oplus_{\alpha\in\Delta_{x}}\mathfrak{g_{\alpha}}\right)$,
the subspace $\mathfrak{h}_{x}=\mathfrak{h\cap\mathfrak{g}}_{x}$
is a Cartan subalgebra of $\mathfrak{g}_{x}$ and 
\begin{equation}
\Delta_{x}=\left\{ \alpha\in\Delta\mid\left(\alpha,\beta\right)=0\ \forall\beta\in B\text{ and \ensuremath{\pm\alpha\not\in}B}\right\} \label{eq:roots of g_x}
\end{equation}
 is the root system of $\mathfrak{g}_{x}$. }

For each finite-dimensional contragredient Lie superalgebra $\mathfrak{g}$
with irreducible Cartan matrix, we can explicitly describe the isomorphism
type of $\mathfrak{g}_{x}$. If $\mathcal{B}=\left\{ \beta_{1},\ldots,\beta_{k}\right\} \in\mathcal{S}$
and $x=x_{\beta_{1}}+\cdots+x_{\beta_{k}}$ for some nonzero $x_{\beta_{i}}\in\mathfrak{g_{\beta_{i}}}$,
then by \cite[Remark 6.4]{DS} we have the following. In particular,
the defect of $\mathfrak{g}_{x}$ equals the defect of $\mathfrak{g}$
minus $k$.
\[
\begin{array}{|c||c|c|c|c|c|c|}
\hline \mathfrak{g} & \mathfrak{gl}\left(m|n\right) & \mathfrak{sl}\left(m|n\right),\ m\neq n & \mathfrak{osp}\left(m|2n\right) & D\left(2,1,\alpha\right) & F_{4} & G_{3}\\
\hline  &  &  &  &  &  & \\
\mathfrak{g}_{x} & \mathfrak{gl}\left(m-k|n-k\right) & \mathfrak{sl}\left(m-k|n-k\right) & \mathfrak{osp}\left(m-2k|2n-2k\right) & \mathbb{C} & \mathfrak{sl}\left(3\right) & \mathfrak{sl}\left(2\right)
\\\hline \end{array}
\]

\begin{rem}
\label{rem:embedding}Note that when $\mathfrak{g}_{x}$ is simple,
the embedding $\mathfrak{g}^{x}\subset\mathfrak{g}$ is determined
by the condition that the root spaces of $\mathfrak{g}_{x}$ are mapped
into the respective root spaces of $\mathfrak{g}$, since in this
case $\mathfrak{h}_{x}\subset\left[\mathfrak{n}_{x}^{+},\mathfrak{n}_{x}^{-}\right]$.
For $\mathfrak{g}=\mathfrak{gl}\left(m,n\right)$, we take the matrix
embedding of $\mathfrak{g}_{x}=\mathfrak{gl}\left(m-k|n-k\right)$
into $\mathfrak{gl}\left(m|n\right)$ which has $2k$ zero rows and
$2k$ zero columns at the locations $r_{i},\ n+s_{i}$, for $i=1,\ldots,k$,
when $B=\left\{ \beta_{i}=\varepsilon_{r_{i}}-\delta_{s_{i}}\right\} _{i=1,\ldots,k}$
is the set of maximal isotropic roots defining $x.$
\end{rem}

\section{The Duflo-{}-Serganova functor and the supercharacter ring}

In this section, we prove that the Duflo\textendash Serganova functor
$DS_{x}:\mathcal{F}_{\mathfrak{g}}\rightarrow\mathcal{F}_{\mathfrak{g}_{x}}$
induces a ring homomorphism $ds_{x}:\,\mathcal{J}_{\mathfrak{g}}\rightarrow\mathcal{J}_{\mathfrak{g}_{x}}$,
and we realize it as a certain evaluation of the functions $f\in\mathcal{J}_{\mathfrak{g}}$
related to the supersymmetry property defining $\mathcal{J}_{\mathfrak{g}}$.

\subsection{The ring homomorphism induced by the Duflo\textendash Serganova functor.\label{subsec:The-DS-ring hom}}

Let $\mathfrak{g}$ be a finite-dimensional contragredient Lie superalgebra
with indecomposable Cartan matrix or let $\mathfrak{\mathfrak{g}=\mathfrak{gl}}\left(m,n\right)$,
and fix a Cartan subalgebra $\mathfrak{h}$ of $\mathfrak{g}$. Let
$B=\left\{ \beta_{1},\ldots,\beta_{k}\right\} \in\mathcal{S_{\mathfrak{g}}}$
, $x\in X_{\mathfrak{g}}$ and $x=x_{\beta_{1}}+\cdots+x_{\beta_{k}}$
for nonzero $x_{\beta_{i}}\in\mathfrak{g_{\beta_{i}}}$. Fix an embedding
$\mathfrak{g}_{x}\subset\mathfrak{g}^{x}\subset\mathfrak{g}$ with
Cartan subalgebra $\mathfrak{h}_{x}=\mathfrak{h}\cap\mathfrak{g}_{x}$
(see Section \ref{subsec:The-Duflo-Serganova-functors}). 

\begin{lem}
\label{thm:well definedness} For $\mathfrak{g}$-modules $M$ and
$N$ we have
\end{lem}
\begin{enumerate}
\item $\operatorname{sch}M_{x}\left(h\right)=\operatorname{sch}M\left(h\right)$
for all $h\in\mathfrak{h}_{x}$;
\item if $\operatorname{sch}M=\operatorname{sch}N$, then $\operatorname{sch}M_{x}=\operatorname{sch}N_{x}$.
\end{enumerate}
\begin{proof}
We have an exact sequence $0\rightarrow\ker_{M}x\rightarrow M\rightarrow xM\rightarrow0$
of $\mathfrak{h}^{x}$-invariant spaces. Thus, $M/\ker_{M}x\cong\Pi\left(xM\right)$
as $\mathfrak{h}^{x}$-modules, where $\Pi$ switches the parity of
a superspace, and so $\operatorname{sch}\left(M/\ker_{M}x\right)\left(h\right)=\operatorname{sch}\Pi\left(xM\right)\left(h\right)$
for all $h\in\mathfrak{h}^{x}$. Hence, for all $h\in\mathfrak{h}_{x}\subset\mathfrak{h}^{x}$
we have that
\[
\begin{alignedat}{1}\operatorname{sch}M\left(h\right) & =\operatorname{sch}\ker x\left(h\right)+\operatorname{sch}\Pi\left(xM\right)\left(h\right)=\operatorname{sch}\ker x\left(h\right)-\operatorname{sch}M\left(h\right)=\operatorname{sch}\left(\ker_{M}x/xM\right)\left(h\right)\\
 & =\operatorname{sch}\left(M_{x}\right)\left(h\right).
\end{alignedat}
\]
\end{proof}
\begin{rem}
The following example shows that Lemma \ref{thm:well definedness}
does not hold if we replace supercharacter by character. It also shows
that the Duflo\textendash Serganova functor is not exact.
\end{rem}
\begin{example}
Let $\mathfrak{g}=\mathfrak{gl}\left(2|1\right)$ with the standard
choice of simple roots $\left\{ \alpha=\varepsilon_{1}-\varepsilon_{2},\beta=\varepsilon_{2}-\delta_{1}\right\} $.
Let $K\left(0\right)$ be the Kac module with highest weight zero,
and denote the highest weight vector by $v_{0}$. Then $K\left(0\right)=\text{span}\left\{ v_{0},f_{\beta}v_{0},f_{\alpha+\beta}v_{0},f_{\beta}f_{\alpha+\beta}v_{0}\right\} $,
where $f_{\beta}\in\mathfrak{g}_{-\beta}$ and $f_{\alpha+\beta}\in\mathfrak{g}_{-\alpha-\beta}$
are nonzero. The maximal submodule of $K\left(0\right)$ is $\overline{K}\left(0\right):=\text{span}\left\{ f_{\beta}v_{0},f_{\alpha+\beta}v_{0},f_{\beta}f_{\alpha+\beta}v_{0}\right\} $,
and the simple quotient of $K(0)$ is isomorphic to the trivial $\mathfrak{g}$-module
$L\left(0\right)$. Clearly, the $\mathfrak{g}$-modules $K\left(0\right)$
and $L\left(0\right)\oplus\overline{K}\left(0\right)$ have the same
character and supercharacter. 

Let us show that for $x=f_{\beta}$, the $\mathfrak{g}_{x}$-modules
$K\left(0\right)_{x}$ and $\left(L\left(0\right)\oplus\overline{K}\left(0\right)\right)_{x}$
have the same supercharacter but not the same character. In this case,
$\mathfrak{g}_{x}=\mathfrak{gl}\left(1|0\right)$. By a direct computation
using (\ref{eq: Mx}) and the basis given above, one can check that
$K(0)_{x}=\left\{ 0\right\} $, $L\left(0\right)_{x}\cong\mathbb{C}_{1|0}$
and $\overline{K}\left(0\right)_{x}\cong\mathbb{C}_{0|1}$, where
$\mathbb{C}_{1|0}$ (resp. $\mathbb{C}_{0|1}$) is the even (resp.
odd) trivial $\mathfrak{g}_{x}$-module. Thus, $\text{ch}K\left(0\right)_{x}=\text{sch}K\left(0\right)_{x}=0$
and $\text{sch}\left(L\left(0\right)\oplus\overline{K}\left(0\right)\right)_{x}=0$,
while $\text{ch}\left(L\left(0\right)\oplus\overline{K}\left(0\right)\right)_{x}=2$. 
\end{example}
\begin{defn}
We define $ds_{x}:\,\mathcal{J}_{\mathfrak{g}}\rightarrow\mathcal{J}_{\mathfrak{g}_{x}}$
on the generators $\left[M\right]\in\mathcal{J}_{\mathfrak{g}}$,
where $M\in\mathcal{F}_{\mathfrak{g}}$, by 
\[
ds_{x}\left(\left[M\right]\right)=\left[DS_{x}\left(M\right)\right],
\]
and we extend linearly to $\mathcal{J}_{\mathfrak{g}}$. 
\end{defn}
It is not difficult to show that $ds_{x}$ is a well-defined linear
map using Lemma \ref{thm:well definedness}. The fact that $ds_{x}$
is a ring homomorphism then follows from the fact that $DS_{x}$ is
a tensor functor. Hence, we have 
\begin{prop}
\label{thm: DS induces a ring homomorphism}Let $\mathfrak{g}$ be
a finite-dimensional contragredient Lie superalgebra, and let $x\in\mathfrak{g}_{\bar{1}}$
nonzero such that $\left[x,x\right]=0$. The functor $DS_{x}:\ \mathcal{F}{}_{\mathfrak{g}}\rightarrow\mathcal{F}_{\mathfrak{g}_{x}}$
induces a ring homomorphism on the corresponding supercharacter rings
$ds_{x}:\,\mathcal{J}_{\mathfrak{g}}\rightarrow\mathcal{J}_{\mathfrak{g}_{x}}$.
\end{prop}
\begin{rem}
The proofs in Section \ref{subsec:The-DS-ring hom} also work for
modules in the BGG Category $\mathcal{O}$, and so the Duflo\textendash Serganova
functor induces a group homomorphism on the quotient of the Grothendieck
group by the parity. However, it is not a ring homomorphism since
Category $\mathcal{O}$ is not closed under tensor products.
\end{rem}

\subsection{Realization of the ring homomorphism.\label{subsec:Realization}}

Given $f\in\mathcal{J}_{\mathfrak{g}}$ we can realize $f:\mathfrak{h}\rightarrow\mathbb{C}$
as a supersymmetric function in the variables $x_{1},\ldots,x_{m},y_{1},\ldots,y_{n}$
with $x_{i}=e^{\varepsilon_{i}}$ and $y_{j}=e^{\delta_{j}}$, using
the supercharacter ring description of Sergeev and Veselov (see Equation
\ref{eq:SV character ring}). (Note that for $F\left(4\right)$ we
take $x_{i}=e^{\frac{1}{2}\varepsilon_{i}}$ and $y_{j}=e^{\frac{1}{2}\delta_{j}}$.) 
\begin{thm}
\label{prop:realization dt} Suppose $ds_{x}:\,\mathcal{J}_{\mathfrak{g}}\rightarrow\mathcal{J}_{\mathfrak{g}_{x}}$
is defined by $x=x_{\beta_{1}}+\cdots+x_{\beta_{k}}$ for nonzero
$x_{\beta_{i}}\in\mathfrak{g_{\beta_{i}}}$, where $B=\left\{ \beta_{1},\ldots,\beta_{k}\right\} \in\mathcal{S_{\mathfrak{g}}}$. 
\end{thm}
\begin{enumerate}
\item Then for any $f\in\mathcal{J}_{\mathfrak{g}}$ ,
\[
ds_{x}\left(f\right)=f\vert{}_{\mathfrak{h}_{x}}.
\]
\item If $B=\left\{ \varepsilon_{1}-\delta_{1}\right\} $, then $ds_{x}\left(f\right)$
is given by substituting $x_{1}=y_{1}$ into $f$, that is, 
\[
ds_{x}f=f\vert_{x_{1}=y_{1}}.
\]
If $\mathfrak{g}=F\left(4\right)$ (resp.\textit{ $D\left(2,1,\alpha\right)$})
and $B=\left\{ \frac{1}{2}\left(\varepsilon_{1}+\varepsilon_{2}+\varepsilon_{3}-\delta_{1}\right)\right\} $
(resp. $B=\left\{ \varepsilon_{1}-\varepsilon_{2}-\varepsilon_{3}\right\} $),
then $ds_{x}f$ is given by substituting $y_{1}=x_{1}x_{2}x_{3}$
(respectively, $x_{1}=x_{2}x_{3}$ ) into $f$ .
\item If $B=\left\{ \beta_{i}=a_{i}\varepsilon_{r_{i}}-b_{i}\delta_{s_{i}}\right\} _{i=1,\ldots,k}$
for some $a_{i},b_{i}\in\left\{ \pm1\right\} $, then $ds_{x}f$ is
given by substituting $x_{r_{i}}^{a_{i}}=y_{s_{i}}^{b_{i}}$ into
$f$, that is,
\[
ds_{x}f=f\vert_{x_{r_{i}}^{a_{i}}=y_{s_{i}}^{b_{i}};\ i=1,\ldots,k}.
\]
\item For any $f\in\mathcal{J}_{\mathfrak{g}}$ , 
\[
ds_{x}\left(f\right)=f\vert_{\beta_{1}=\cdots=\beta_{k}=0}.
\]
\end{enumerate}
\begin{proof}
It suffices to prove (1) for a spanning set of $\mathcal{J}_{\mathfrak{g}}$.
Suppose $\left[M\right]\in\mathcal{K_{\mathfrak{g}}}$ corresponds
to a module $M\in\mathcal{F}_{\mathfrak{g}}$. By Lemma \ref{thm:well definedness},
we have
\[
ds_{x}\left(\left[M\right]\right)=\left[DS_{x}\left(M\right)\right]=\operatorname{sch}M_{x}=\left(\operatorname{sch}M\right)\vert_{\mathfrak{h}_{x}}=\left[M\right]\vert_{\mathfrak{h}_{x}}\in\mathcal{J}_{\mathfrak{g}_{x}}.
\]
Hence, $ds_{x}\left(f\right)=f\vert{}_{\mathfrak{h}_{x}}$ for any
$f\in\mathcal{J}_{\mathfrak{g}}$.

To prove (2), fix $f\in\mathcal{J}_{\mathfrak{g}}$, and suppose that
$x\in\mathfrak{g}_{\beta}$. If $\beta=\varepsilon_{1}-\delta_{1}$,
then the evaluation $f_{x_{1}=y_{1}=t}$ is well defined and independent
of $t$ due to the supersymmetry property of $f\in\mathcal{J}_{\mathfrak{g}}$.
Thus, 
\[
f\vert_{x_{1}=y_{1}=t}:=f\left(t,x_{2},\ldots,x_{m}\mid t,y_{2},\ldots,y_{n}\right)
\]
 is equal to the restriction of $f$ to the hyperplane $x_{1}-y_{1}=0$.
Since $\mathfrak{h}_{x}\subset\mathfrak{h}^{x}=\left\{ h\in\mathfrak{h}\,\mid\,\beta\left(h\right)=0\right\} $
belongs to the hyperplane $x_{1}-y_{1}=0$, we have proven that $ds_{x}f=f\vert_{\mathfrak{h}_{x}}=f\vert_{x_{1}=y_{1}}$.
The cases $\beta=\varepsilon_{1}+\varepsilon_{2}+\varepsilon_{3}-\delta_{1}$
and $\beta=\varepsilon_{1}-\varepsilon_{2}-\varepsilon_{3}$ are similar. 

Now (3) can be proven using arguments similar to that of (2) and the
fact that
\[
\mathfrak{h}^{x}=\left\{ h\in\mathfrak{h}\,\mid\,\beta\left(h\right)=0\text{ for all }\beta\in B\right\} .
\]
Finally, (4) follows from (3) since if $\beta_{i}=a_{i}\varepsilon_{r_{i}}-b_{i}\delta_{r_{i}}$,
then $\beta_{i}=0$ if and only if $x_{r_{i}}^{a_{i}}y_{r_{i}}^{-b_{i}}=e^{\beta_{i}}=1$
if and only if $x_{r_{i}}^{a_{i}}=y_{r_{i}}^{b_{i}}$. 
\end{proof}
\begin{cor}
\label{lem:composition of ds}If $x=x_{\beta_{1}}+\ldots+x_{\beta_{k}}$
where $x_{\beta_{i}}\in\mathfrak{g}_{\beta_{i}}$ and $B=\left\{ \beta_{1},\ldots,\beta_{k}\right\} \in\mathcal{S}$,
then for all $f\in\mathcal{J}_{\mathfrak{g}}$ 
\[
ds_{x}\left(f\right)=ds_{x_{\beta_{1}}}\circ\ldots\circ ds_{x_{\beta_{k}}}\left(f\right).
\]
\end{cor}

\section{The kernel of the ring homomorphism}

In this section, we give a $\mathbb{Z}$-basis for the kernel of $ds_{x}$
when $x\in\mathfrak{g}_{\beta}$ is a root vector of an isotropic
root $\beta$ for the Lie superalgebra $\mathfrak{g}$. Our basis
is given by elements of the following form.
\begin{defn}
\label{def: k lambda}For each $\lambda\in P_{\bar{0}}$, we define
\[
k\left(\lambda\right):=R^{-1}\cdot\sum_{w\in W}\left(-1\right)^{l\left(w\right)+p\left(w\left(\rho\right)-\rho\right)}e^{w\left(\lambda+\rho\right)-\rho}.
\]
\end{defn}
Here $p\left(w\left(\rho\right)-\rho\right)$ denotes the parity of
$w\left(\rho\right)-\rho$, which is well-defined since $w\left(\rho\right)-\rho\in Q$.
Note that the element $w\left(\rho\right)-\rho$ may be odd, e.g.
in $\mathfrak{osp}\left(1|2\right)$.

For each $\lambda\in P_{\bar{0}}^{+}$, the expression $k\left(\lambda\right)$
is in $\mathbb{Z}\left[P_{\bar{0}}\right]^{W}$ since it is the product
of the $W$-invariant polynomial $e^{\rho_{\bar{1}}}R_{\bar{1}}$
and the character of a finite-dimensional $\mathfrak{g}_{\bar{0}}$-module
given by the Weyl character formula. Moreover, since the evaluation
$k\left(\lambda\right)\mid_{\beta=0}$ equals zero for any $\beta\in\Delta_{iso}$
we have that $k\left(\lambda\right)\in\mathcal{J}_{\mathfrak{\mathfrak{g}}}$.
It is clear that $k\left(\lambda\right)$ is in the kernel of $ds_{x}$
for any $x\in\mathfrak{g}_{\beta}$, since $ds_{x}\left(R_{\bar{1}}\right)=0$.

For Lie superalgebras of Type I with the distinguished choice of simple
roots, $k\left(\lambda\right)$ is the supercharacter of a Kac module
when $\lambda\in P_{\bar{0}}^{+}$, whereas in Type II, $k\left(\lambda\right)$
is a virtual supercharacter. Similar virtual characters were used
by Gruson and Serganova in \cite{GS} to study the character formula
of simple modules over orthosymplectic Lie superalgebras. 

We need the following definition to prove the main result in this
section for Lie superalgebras of Type II. 
\begin{defn}
Given a finite-dimensional Lie superalgebra $\mathfrak{g}$ with root
system $\Delta$, we define a Lie algebra $\widetilde{\mathfrak{g}}$
as follows. We let $\widetilde{\Delta}$ be the root system with positive
even roots given by
\[
\widetilde{\Delta}^{+}:=\left\{ \alpha\in\Delta_{\bar{0}}^{+}\mid\frac{\alpha}{2}\notin\Delta_{\bar{1}}\right\} \cup\left\{ \alpha\in\Delta_{\bar{1}}^{+}\,\mid\,\alpha\not\in\Delta_{iso}\right\} ,
\]
and we let $\tilde{\frak{\mathfrak{g}}}$ be the semisimple Lie algebra
with root system $\widetilde{\Delta}$. If $\Delta_{\bar{1}}=\Delta_{iso}$,
then $\widetilde{\mathfrak{g}}=\mathfrak{g}_{\bar{0}}$. If $\mathfrak{g}=B\left(m|n\right)$,
$G\left(3\right)$ then $\widetilde{\mathfrak{g}}\cong B_{m}\times B_{n}$,
$G_{2}\times A_{1}$, respectively. We set $\widetilde{\rho}:=\frac{1}{2}\sum_{\alpha\in\widetilde{\Delta}^{+}}\alpha$.
Note that $\rho=\widetilde{\rho}-\rho_{iso}$, since $\beta\in\Delta_{iso}$
if and only if $\beta\in\Delta_{\bar{1}}$ but $2\beta\not\in\Delta_{\bar{0}}$.
Let $P_{\widetilde{\mathfrak{g}}}^{+}$ denote the set of dominant
integral weights of $\widetilde{\mathfrak{g}}.$ Then $P_{\bar{0}}^{+}\subset P_{\widetilde{\mathfrak{g}}}^{+}$
and the Weyl group of $\widetilde{\mathfrak{g}}$ is isomorphic to
the Weyl group of $\mathfrak{g}_{\bar{0}}$. We extend the definition
of $k\left(\lambda\right)$ to $\lambda\in P_{\widetilde{\mathfrak{g}}}$
by letting
\[
k\left(\lambda\right):=R^{-1}\cdot\sum_{w\in W}\left(-1\right)^{l\left(w\right)+p\left(w\left(\lambda+\rho\right)-\rho\right)}e^{w\left(\lambda+\rho\right)-\rho}.
\]
\end{defn}
We have the following lemma.
\begin{lem}
\label{lem: linear independence}The set $\left\{ k(\mu)\mid\mu\in P_{\widetilde{\mathfrak{g}}}^{+}+\rho_{iso}\right\} $
is linearly independent.
\end{lem}
\begin{proof}
To prove linear independence we consider a completion of $\mathbb{Z}\left[P_{\widetilde{\mathfrak{g}}}\right]$,
where we allow expansions in the domain $|e^{-\alpha}|<1$ for $\alpha\in\widetilde{\Delta}^{+}$.
Note that in this completion, $R^{-1}=\sum_{\nu\in-Q_{\mathfrak{\widetilde{\mathfrak{g}}}}^{+}}b_{\nu}e^{\nu}$
for some $b_{\nu}\in\mathbb{Z}$. For each $\mu\in P_{\widetilde{\mathfrak{g}}}^{+}+\rho_{iso}$,
we will show that $\mu+\rho$ is a strictly dominant element of $P_{\widetilde{\mathfrak{g}}}$,
that is, $w\left(\mu+\rho\right)<\mu+\rho$ for $w\in W$, $w\ne1$.
Indeed, if $\mu\in P_{\tilde{\mathfrak{g}}}^{+}+\rho_{iso}$, then
$\mu+\rho=\lambda+\widetilde{\rho}$ for some $\lambda\in P_{\tilde{\mathfrak{g}}}^{+}$.
Since $\lambda+\widetilde{\rho}$ is strictly dominant with respect
to $\tilde{\mathfrak{g}}$, it is also strictly dominant with respect
to $\mathfrak{g}_{\bar{0}}$ and the claim follows. Thus,
\[
k\left(\mu\right)=e^{\mu}+\sum_{\nu\in\mu-Q_{\mathfrak{\widetilde{\mathfrak{g}}}}^{+}}a_{\nu}e^{\nu}
\]
and linear independence follows. 
\end{proof}
\begin{rem}
\label{rem:get rid of rho_iso shift}Note that if one takes the distinguished
choice of simple roots for $\mathfrak{gl}\left(m,n\right)$, then
$P^{+}=P^{+}+\rho_{iso}$, since in this case $\left(\rho_{iso},\alpha\right)=0$
for every even root $\alpha$.
\end{rem}
The following lemma is used in the proof of the main theorem of this
section.

\begin{lem}
\label{lem:nonisotropic roots}For each $\mu\in P_{\widetilde{\mathfrak{g}}}^{+}$,
we have $k\left(\mu+\rho_{iso}\right)=e^{\rho_{iso}}\prod_{\alpha\in\Delta_{iso}^{+}}\left(1-e^{-\alpha}\right)\cdot\operatorname{ch}L_{\tilde{\mathfrak{g}}}\left(\mu\right).$
\end{lem}
\begin{proof}
For any element $g\in\mathbb{Z}\left[P_{\widetilde{\mathfrak{g}}}\right]$
with $\operatorname{Supp}g\subset\mu+Q_{\mathfrak{\widetilde{\mathfrak{g}}}}$,
we write $g=\sum_{\lambda\in Q_{\mathfrak{\widetilde{\mathfrak{g}}}}}c_{\mu+\lambda}e^{\mu+\lambda}$,
and we define $\bar{g}=\sum_{\lambda\in Q_{\mathfrak{\mathfrak{\widetilde{\mathfrak{g}}}}}}\left(-1\right)^{p\left(\lambda\right)}c_{\mu+\lambda}e^{\mu+\lambda}$,
where $p:Q_{\mathfrak{\mathfrak{\widetilde{\mathfrak{g}}}}}\rightarrow\mathbb{Z}_{2}$
is the parity function. Clearly, this operation is an involution.
So we have that

\[
\]

$\overline{e^{\rho_{iso}}\prod_{\alpha\in\Delta_{iso}^{+}}\left(1-e^{-\alpha}\right)\cdot\operatorname{ch}L_{\tilde{\mathfrak{g}}}\left(\mu\right)}=$

\begin{eqnarray*}
 & = & \left(-1\right)^{p\left(\rho_{iso}\right)}e^{\rho_{iso}}\prod_{\alpha\in\Delta_{iso}^{+}}\left(1+e^{-\alpha}\right)\cdot\operatorname{sch}L_{\tilde{\mathfrak{g}}}\left(\mu\right)\\
 & = & \left(-1\right)^{p\left(\rho_{iso}\right)}e^{\rho_{iso}}\prod_{\alpha\in\Delta_{iso}^{+}}\left(1+e^{-\alpha}\right)\frac{\sum_{w\in W}\left(-1\right)^{l\left(w\right)+p\left(w\left(\mu+\widetilde{\rho}\right)-\widetilde{\rho}\right)}e^{w\left(\mu+\widetilde{\rho}\right)-\widetilde{\rho}}}{\prod_{\alpha\in\widetilde{\Delta}_{\bar{0}}^{+}}\left(1-e^{-\alpha}\right)}\\
 & = & \prod_{\alpha\in\Delta_{\bar{1}}^{+}}\left(1+e^{-\alpha}\right)\cdot\frac{\sum_{w\in W}\left(-1\right)^{l\left(w\right)+p\left(w\left(\mu+\rho_{iso}+\rho\right)-\rho\right)}e^{w\left(\left(\mu+\rho_{iso}\right)+\widetilde{\rho}-\rho_{iso}\right)-\widetilde{\rho}+\rho_{iso}}}{\prod_{\alpha\in\widetilde{\Delta}_{\bar{0}}^{+}}\left(1-e^{-\alpha}\right)\prod_{\alpha\in\Delta_{\bar{1}}^{+}\backslash\Delta_{iso}^{+}}\left(1+e^{-\alpha}\right)}\\
 & = & {\color{black}\prod_{\alpha\in\Delta_{\bar{1}}^{+}}\left(1+e^{-\alpha}\right)\frac{\sum_{w\in W}\left(-1\right)^{l\left(w\right)+p\left(w\left(\left(\mu+\rho_{iso}\right)+\rho\right)-\rho\right)}e^{w\left(\left(\mu+\rho_{iso}\right)+\rho\right)-\rho}}{\prod_{\alpha\in\Delta_{\bar{0}}^{+}}\left(1-e^{-\alpha}\right)}}\\
{\color{black}{\color{red}}} & = & \overline{k\left(\mu+\rho_{iso}\right)},
\end{eqnarray*}
and hence the claim follows. 
\end{proof}
We have the following theorem.
\begin{thm}
\label{thm:kernel}If $\beta$ is an odd isotropic root and $x\in\mathfrak{g}_{\beta}$,
then the set
\begin{equation}
\left\{ k(\lambda)\mid\lambda\in P_{\bar{0}}^{+}+\rho_{iso}\right\} \label{eq:kernel basis}
\end{equation}
 is a $\mathbb{Z}$-basis for the kernel of $ds_{x}:\mathcal{J}_{\mathfrak{g}}\rightarrow\mathcal{J}_{\mathfrak{g}_{x}}$.
\end{thm}
\begin{proof}
Linear independence of the set (\ref{eq:kernel basis}) follows from
Lemma \ref{lem: linear independence} since $P_{\bar{0}}^{+}\subset P_{\tilde{\mathfrak{g}}}^{+}$.
So it only remains to show that the set (\ref{eq:kernel basis}) spans
the kernel of $ds_{x}:\mathcal{J}_{\mathfrak{g}}\rightarrow\mathcal{J}_{\mathfrak{g}_{x}}$.

Let $f\in\mathcal{J}_{\mathfrak{g}}$ such that $ds_{x}\left(f\right)=0$.
According to Theorem \ref{prop:realization dt}, this means that the
restriction of $f$ to the hyperplane $\beta=0$ is zero, or equivalently,
substituting $e^{-\beta}=1$ yields zero. Hence, $f$ is divisible
by $\left(1-e^{-\beta}\right)$. Since $f$ is $W$-invariant and
$W\beta=\Delta_{iso}$, it follows that $f$ is divisible by $e^{\rho_{iso}}\prod_{\alpha\in\Delta_{iso}^{+}}\left(1-e^{-\alpha}\right)$.

Write 
\[
f=e^{\rho_{iso}}\prod_{\alpha\in\Delta_{iso}^{+}}\left(1-e^{-\alpha}\right)\cdot g.
\]
Then $g$ is a $W$-invariant element of $\mathbb{Z}\left[P_{\bar{0}}\right]$,
since both $f$ and $e^{\rho_{iso}}\prod_{\alpha\in\Delta_{iso}^{+}}\left(1-e^{-\alpha}\right)$
are $W$-invariant.

Case 1: First suppose that $\mathfrak{g}$ does not have non-isotropic
roots, then $\Delta_{iso}^{+}=\Delta_{\bar{1}}^{+}$ and $\rho_{iso}=\rho_{\bar{1}}$.
By the theory of symmetric functions,
\[
g=\sum_{\mu\in P_{\bar{0}}^{+}}^{\text{finite}}a_{\mu}\operatorname{ch}L_{\mathfrak{g}_{\bar{0}}}\left(\mu\right),
\]
 for some $a_{\mu}\in\mathbb{Z}$, where $P_{\bar{0}}^{+}$ is the
set of highest weights of finite-dimensional $\mathfrak{g}_{\bar{0}}$-modules
(see for example \cite{M}).

By the Weyl character formula for semisimple Lie algebras, we have
that
\begin{eqnarray*}
f & = & e^{\rho_{\bar{1}}}R_{\bar{1}}\cdot g\\
 & = & e^{\rho_{\bar{1}}}R_{\bar{1}}\sum_{\mu\in P_{\bar{0}}^{+}}^{\mbox{ }}a_{\mu}\operatorname{ch}L_{\mathfrak{g}_{\bar{0}}}\left(\mu\right)\\
 & = & e^{\rho_{\bar{1}}}R_{\bar{1}}\sum_{\mu\in P_{\bar{0}}^{+}}a_{\mu}\frac{\sum_{w\in W}\left(-1\right)^{l\left(w\right)}e^{w\left(\mu+\rho_{0}\right)}}{e^{\rho_{\bar{0}}}R_{\bar{0}}}\\
 & = & e^{\rho_{\bar{1}}}R_{\bar{1}}\sum_{\lambda\in P_{\bar{0}}^{+}+\rho_{\bar{1}}}b_{\lambda}\frac{\sum_{w\in W}\left(-1\right)^{l\left(w\right)}e^{w\left(\lambda+\rho_{0}-\rho_{1}\right)}}{e^{\rho_{\bar{0}}}R_{\bar{0}}}\\
 & = & \sum_{\lambda\in P_{\bar{0}}^{+}+\rho_{\bar{1}}}b_{\lambda}k\left(\lambda\right),
\end{eqnarray*}
where $b_{\lambda}:=a_{\lambda-\rho_{\bar{1}}}$. For each $w\in W$,
the parity of $w\left(\rho\right)$ equals the parity of $\rho$,
since $\rho\in P_{\bar{0}}$. Hence, the last equality follows.

Case 2: Suppose that $\mathfrak{g}$ has non-isotropic roots. Since
$P_{\bar{0}}\subset P_{\widetilde{\mathfrak{g}},}$, by the theory
of characters of Lie algebras
\[
g=\sum_{\mu\in P_{\widetilde{\mathfrak{g}}}^{+}}^{\text{finite}}a_{\mu}\mbox{ch}L_{\widetilde{\mathfrak{g}}}\left(\mu\right)
\]
 for some $a_{\mu}\in\mathbb{Z}$. By Lemma \ref{lem:nonisotropic roots},
we have that
\begin{eqnarray}
f & = & e^{\rho_{iso}}\prod_{\alpha\in\Delta_{iso}^{+}}\left(1-e^{-\alpha}\right)\cdot g\nonumber \\
 & = & e^{\rho_{iso}}\prod_{\alpha\in\Delta_{iso}^{+}}\left(1-e^{-\alpha}\right)\sum_{\mu\in P_{\widetilde{\mathfrak{g}}}^{+}}^{\mbox{ }}a_{\mu}\mbox{ch}L_{\tilde{\mathfrak{g}}}\left(\mu\right)\nonumber \\
 & = & \sum_{\mu\in P_{\widetilde{\mathfrak{g}}}^{+}}a_{\mu}\cdot e^{\rho_{iso}}\prod_{\alpha\in\Delta_{iso}^{+}}\left(1-e^{-\alpha}\right)\cdot\mbox{ch}L_{\tilde{\mathfrak{g}}}\left(\mu\right)\nonumber \\
 & = & \sum_{\mu\in P_{\widetilde{\mathfrak{g}}}^{+}}a_{\mu}\cdot k\left(\mu+\rho_{iso}\right)\nonumber \\
{\color{black}{\color{red}}} & {\color{black}=} & {\color{black}\sum_{\lambda\in P_{\widetilde{\mathfrak{g}}}^{+}+\rho_{iso}}b_{\lambda}k\left(\lambda\right)}\label{eq:sum is finite}
\end{eqnarray}
where $b_{\lambda}:=a_{\lambda-\rho_{iso}}$. We are left to show
that $b_{\lambda}=0$ for $\lambda\not\in P_{\bar{0}}^{+}+\rho_{iso}$.
Since $\operatorname{Supp}f\subset P_{\bar{0}}$, $\operatorname{Supp}k\left(\lambda\right)\subset P_{\bar{0}}$,
the elements $k\left(\lambda\right)$ for $\mu\in P_{\widetilde{\mathfrak{g}}}^{+}+\rho_{iso}$
are linearly independent and the sum in (\ref{eq:sum is finite})
is finite, we conclude that
\[
f=\sum_{\lambda\in P_{\bar{0}}^{+}+\rho_{iso}}b_{\lambda}k\left(\lambda\right).
\]
\end{proof}
\begin{cor}
Let $G$ be one of the Lie supergroups $SL\left(m|n\right)$, $m\neq n,$
$GL\left(m|n\right)$ or $SOSP\left(m|2n\right)$,\textup{\emph{ and
let $\mathfrak{g}$ be the corresponding Lie superalgebra. Let}} $\beta$
be an odd isotropic root\textup{\emph{, }}$x\in\mathfrak{g}_{\beta}$
and let $DS_{x}:\mathcal{F}_{G}\rightarrow\mathcal{F}_{G_{x}}$ \textup{\emph{be
the }}Duflo-Serganova functor\textup{\emph{ from the category }}$\mathcal{F}_{G}$\textup{\emph{
of finite-dimensional}} $G$\textup{\emph{-modules to the category
}}$\mathcal{F}_{G_{x}}$\textup{\emph{ of finite-dimensional}} $G_{x}$\textup{\emph{-modules,
where}} $G_{x}$ denotes the Lie supergroup corresponding to the Lie
superalgebra $\mathfrak{g}_{x}$\textup{\emph{. T}}hen the kernel\textup{\emph{
of the induced }}ring homomorphism $ds_{x}:\mathcal{\mathcal{J}}_{G}\rightarrow\mathcal{\mathcal{J}}_{G_{x}}$
has a $\mathbb{Z}$-basis \textup{
\[
\left\{ k(\lambda)\mid\lambda\in P_{G}^{+}+\rho_{iso}\right\} ,
\]
}where $P_{G}^{+}$ is the set of highest weights for finite-dimensional
$G$-modules.
\end{cor}
\begin{proof}
Let $P_{G}\subset P_{\bar{0}}$ be the sublattice of integral weights
of finite-dimensional $G_{\bar{0}}$-modules. Then for $G=GL\left(m|n\right)$
or $SOSP\left(m|2n\right)$
\[
P_{G}=\left\{ \sum_{i=1}^{m}\lambda_{i}\varepsilon_{i}+\sum_{j=1}^{n}\mu_{j}\delta_{j}\mid\lambda_{i},\mu_{j}\in\mathbb{Z}\right\} ,
\]
and the supercharacter ring for the category of finite-dimensional
$G$-modules $\mathcal{F}_{G}$ is 
\[
\mathcal{J}_{G}=\left\{ f\in\mathbb{Z}\left[x_{1}^{\pm1},\ldots,x_{m}^{\pm1},y_{1}^{\pm1},\ldots,y_{n}^{\pm1}\right]^{W}\,\mid\,y_{j}\frac{\partial f}{\partial y_{j}}+x_{i}\frac{\partial f}{\partial x_{i}}\in\left\langle y_{j}-x_{i}\right\rangle \right\} 
\]
as shown in \cite[Section 7]{SV} (note that this ring is denoted
by $J\left(\mathfrak{g}\right)_{0}$ in \cite{SV}). If $G=SL\left(m|n\right)$,
$m\neq n,$ then
\[
P_{G}=\left\{ \sum_{i=1}^{m}\lambda_{i}\varepsilon_{i}+\sum_{j=1}^{n}\mu_{j}\delta_{j}\mid\lambda_{i},\mu_{j}\in\mathbb{Z},\ \sum_{i=1}^{m}\lambda_{i}-\sum_{j=1}^{n}\mu_{j}=0\right\} ,
\]
and the supercharacter ring for the category of finite-dimensional
$G$-modules $\mathcal{F}_{G}$ is 
\[
\mathcal{J}_{G}=\left\{ f\in\mathbb{Z}\left[x_{1}^{\pm1},\ldots,x_{m}^{\pm1},y_{1}^{\pm1},\ldots,y_{n}^{\pm1}\right]^{W}/\left\langle x_{1}\cdots x_{m}-y_{1}\cdots y_{n}\right\rangle \mid\,y_{j}\frac{\partial f}{\partial y_{j}}+x_{i}\frac{\partial f}{\partial x_{i}}\in\left\langle y_{j}-x_{i}\right\rangle \right\} 
\]
as shown in \cite[Section 7]{SV}. Since in both cases $\mathcal{J}_{G}=\mathcal{J}_{\mathfrak{g}}\cap\mathbb{Z}\left[P_{G}\right]$,
the kernel of the homomorphism $ds_{x}:\mathcal{\mathcal{J}}_{G}\rightarrow\mathcal{\mathcal{J}}_{G_{x}}$
equals $\operatorname{Ker}_{G}ds_{x}=\operatorname{Ker}_{\mathfrak{g}}ds_{x}\cap\mathbb{Z}\left[P_{G}\right],$
where $\operatorname{Ker}_{\mathfrak{g}}ds_{x}$ is the kernel of
the corresponding homomorphism $ds_{x}:\mathcal{J}_{\mathfrak{g}}\rightarrow\mathcal{J}_{\mathfrak{g}_{x}}$.
It follows from the linear independence of the elements $k\left(\lambda\right)$
and the fact that $\lambda\in P_{G}$ if and only if $\operatorname{Supp}k\left(\lambda\right)\in P_{G}$,
that $\operatorname{Ker}_{G}ds_{x}=\mathrm{span}\left\{ k\left(\lambda\right)\mid\lambda\in P_{G}+\rho_{iso}\right\} $.
Since $P_{G}^{+}=P_{\mathfrak{g}}^{+}\cap P_{G}$, the claim follows.
\end{proof}
\begin{rem}
On the level of categories, it was shown in \cite{BKN} that a module
$M$ over a Type I finite-dimensional contragredient Lie superalgebra
has a filtration of Kac modules (resp. dual Kac modules) if and only
if $DS_{x}\left(M\right)=0$ for all $x\in X_{\mathfrak{g}}^{-}$
(resp. $x\in X_{\mathfrak{g}}^{+}$), where $X_{\mathfrak{g}}^{\pm}=X_{\mathfrak{g}}\cap\mathfrak{n}^{\pm}$
and $\mathfrak{g}=\mathfrak{n}^{-}\oplus\mathfrak{h}\oplus\mathfrak{n}^{+}$
is the triangular decomposition with respect to the distinguished
choice of simple roots. 
\end{rem}

\section{The image of the ring homomorphism}

\subsection{Image of $ds_{x}$ for classical Lie superalgebras}

Let $\mathfrak{g}$ be one of the Lie superalgebras: $\mathfrak{sl}\left(m|n\right)$,
$m\neq n$, $\mathfrak{gl}\left(m|n\right)$ and $\mathfrak{osp}\left(m|2n\right)$.
In this section, we describe the image of $ds_{x}$ for every $x\in X_{\mathfrak{g}}$.
We use the realization of $ds_{x}$ given in Theorem \ref{prop:realization dt},
and the explicit description of the supercharacter rings given by
Sergeev and Veselov in \cite[Section 7]{SV}. 
\begin{thm}
\label{thm:surjectivity}Let $G$ be one of the Lie supergroups $SL\left(m|n\right)$,
$m\neq n,$ $GL\left(m|n\right)$ or $OSP\left(m,2n\right)$\textup{\emph{
and let $\mathfrak{g}$ be the corresponding Lie superalgebra. For
any $x\in X_{\mathfrak{g}}$}}, the Duflo-Serganova functor $DS_{x}:\mathcal{F}_{G}\rightarrow\mathcal{F}_{G_{x}}$
\textup{\emph{from the category }}$\mathcal{F}_{G}$\textup{\emph{
of finite-dimensional}} $G$\textup{\emph{-modules to the category
}}$\mathcal{F}_{G_{x}}$\textup{\emph{ of finite-dimensional}} $G_{x}$\textup{\emph{-modules}}
induces a surjective ring homomorphism on the corresponding supercharacter
rings $ds_{x}:\mathcal{\mathcal{J}}_{G}\rightarrow\mathcal{\mathcal{J}}_{G_{x}}$.
\end{thm}
\begin{proof}
We will use Corollary \ref{lem:composition of ds} to reduce to the
case that $x\in\mathfrak{g}_{\beta}$ for some isotropic root $\beta$.
Using the realization of $ds_{x}$ given in Theorem \ref{prop:realization dt}
we will show that $ds_{x}$ transfers a certain set of generators
of the supercharacter ring $\mathcal{K}_{G}$ to a set of generators
of the supercharacter ring $\mathcal{J}_{G_{x}}$. We use the same
set of generators of $\mathcal{J}_{G}$ that Sergeev and Veselov used
to give explicit descriptions of supercharacter rings over basic Lie
superalgebras and their corresponding Lie supergroups \cite[Section 7]{SV}.
\begin{lyxlist}{00.00.0000}
\item [{$GL\left(m,n\right)$:}] The supercharacter ring of $GL\left(m,n\right)$
is generated by $\frac{x_{1}\cdots x_{m}}{y_{1}\cdots y_{n}}$, $\frac{y_{1}\cdots y_{n}}{x_{1}\cdots x_{m}}$,
$h_{k}\left(x_{1},\ldots,x_{m},y_{1},\ldots,y_{n}\right)$, $h_{k}\left(x_{1}^{-1},\ldots,x_{m}^{-1},y_{1}^{-1},\ldots,y_{n}^{-1}\right)$,
$k\in\mathbb{Z}_{>0}$, where 
\begin{equation}
\chi_{G}\left(t\right)=\frac{\prod_{i=1}^{m}\left(1-x_{i}t\right)}{\prod_{j=1}^{n}\left(1-y_{j}t\right)}=\sum_{k=0}^{\infty}h_{k}\left(x_{1},\ldots,x_{m},y_{1},\ldots,y_{n}\right)t^{k}.\label{eq: hk for GL}
\end{equation}
\item [{$SL\left(m,n\right)$,}] $m\neq n$: The supercharacter ring of
$SL\left(m,n\right)$, $m\neq n$, is generated by $h_{k}\left(x_{1},\ldots,x_{m},y_{1},\ldots,y_{n}\right)$,
$h_{k}\left(x_{1}^{-1},\ldots,x_{m}^{-1},y_{1}^{-1},\ldots,y_{n}^{-1}\right)$,
$k\in\mathbb{Z}_{>0}$, where $h_{k}$ is given by (\ref{eq: hk for GL}).
\item [{$OSP(2m+1,2n)$:}] The supercharacter ring of $OSP(2m+1,2n)$ is
generated by $h_{k}\left(x_{1},\ldots,x_{m},y_{1},\ldots,y_{n}\right)$,
$k\in\mathbb{Z}_{>0}$ where
\[
\chi_{G}\left(t\right)=\frac{\prod_{j=1}^{n}\left(1-y_{j}t\right)\left(1-y_{j}^{-1}t\right)}{\left(1-t\right)\prod_{i=1}^{m}\left(1-x_{i}t\right)\left(1-x_{i}^{-1}t\right)}=\sum_{k=0}^{\infty}h_{k}\left(x_{1},\ldots,x_{m},y_{1},\ldots,y_{n}\right)t^{k}.
\]
.
\item [{$OSP(2,2n)$:}] The supercharacter ring of $OSP(2,2n)$ is generated
by $h_{k}\left(x_{1},y_{1},\ldots,y_{n}\right)$, $k\in\mathbb{Z}_{>0}$,
where
\[
\chi_{G}\left(t\right)=\frac{\prod_{i=1}^{m}\left(1-y_{j}t\right)\left(1-y_{j}^{-1}t\right)}{\left(1-x_{1}t\right)\left(1-x_{1}^{-1}t\right)}=\sum_{k=0}^{\infty}h_{k}\left(x_{1},y_{1},\ldots,y_{n}\right)t^{k}.
\]
\item [{$OSP(2m,2n)$,}] $m\ge2$: The supercharacter ring of $OSP(2m,2n)$
is generated by $h_{k}\left(x_{1},\ldots,x_{m},y_{1},\ldots,y_{n})\right)$,
$k\in\mathbb{Z}_{>0}$, where
\[
\chi_{G}\left(t\right)=\frac{\prod_{p=1}^{n}\left(1-y_{j}t\right)\left(1-y_{j}^{-1}t\right)}{\prod_{i=1}^{m}\left(1-x_{i}t\right)\left(1-x_{i}^{-1}t\right)}=\sum_{k=0}^{\infty}h_{k}\left(x_{1},\ldots,x_{m},y_{1},\ldots,y_{n}\right)t^{k}.
\]
\end{lyxlist}
By Theorem \ref{prop:realization dt}, $ds_{x}\left(h_{k}^{\mathfrak{g}}\right)=\left(h_{k}^{\mathfrak{g}}\right)\vert_{\beta=0}$.
Since $\chi_{G}$ is $W$-invariant and $W\beta=\Delta_{iso}$ for
any $\beta\in\Delta_{iso}$, it suffices to consider the case that
$\beta=\varepsilon_{1}-\delta_{1}$. In this case, $\beta=0$ if and
only if $x_{1}=y_{1}$. It is not difficult to check that $\chi_{G}\left(t\right)\vert_{x_{1}=y_{1}}=\chi_{G_{x}}$,
and hence $ds_{x}\left(h_{k}^{\mathfrak{g}}\right)=h_{k}^{\mathfrak{g}_{x}}$.
Thus, all the generators of $\mathcal{J}_{G_{x}}$ are in the image
of $ds_{x}$. 

The general case for arbitrary $x\in X_{\mathfrak{g}}$ now follows
from Corollary \ref{lem:composition of ds}, since the composition
of surjective maps is surjective.
\end{proof}
\begin{prop}
Let $\mathfrak{\mathfrak{g}=sl}\left(m|n\right)$, $m\neq n$, or
$\mathfrak{\mathfrak{g}=osp}\left(m|2n\right)$.\textup{\emph{ Then
for any $x\in X_{\mathfrak{g}}$}}, the image of $ds_{x}:\mathcal{J}_{\mathfrak{g}}\rightarrow\mbox{\ensuremath{\mathcal{J}}}_{\mathfrak{g}_{x}}$
is the supercharacter ring $\mathcal{J}_{G_{x}}$ of the Lie supergroup
$G_{x}$.
\end{prop}
\begin{proof}
We use Theorem \ref{thm:surjectivity} and Theorem \ref{prop:realization dt},
together with the description of the rings $\mathcal{J}_{\mathfrak{g}}$
given by Sergeev and Veselov in \cite{SV} to prove that the image
of the map $ds_{x}:\mathcal{J}_{\mathfrak{g}}\rightarrow\mbox{\ensuremath{\mathcal{J}}}_{\mathfrak{g}_{x}}$
equals $\mathcal{J}_{G_{x}}$ in the case that $x\in\mathfrak{g}_{\beta}$
is an isotropic root $\beta$. The claim for any element $x\in X_{\mathfrak{g}}$
then follows Corollary \ref{lem:composition of ds}.

The supercharacter ring of $\mathfrak{g}=\mathfrak{sl}\left(m|n\right)$,
$m\neq n$, is $\mathcal{J}_{\mathfrak{g}}=\mathcal{J}_{G}\oplus\bigoplus_{a\in\mathbb{C}/\mathbb{Z}}J\left(\mathfrak{g}\right)_{a}$,
where 
\[
J\left(\mathfrak{g}\right)_{a}=\left(x_{1}\cdots x_{n}\right)^{a}\Pi_{i,j}\left(1-x_{i}y_{j}^{-1}\right)\mathbb{Z}\left[x^{\pm1},y^{\pm1}\right]_{0}^{S_{m}\times S_{n}},
\]
 and $\mathbb{Z}\left[x^{\pm1},y^{\pm1}\right]_{0}^{S_{m}\times S_{n}}$
is the quotient of the ring $\mathbb{Z}\left[x_{1}^{\pm1},\ldots,x_{m}^{\pm1},y_{1}^{\pm1},\ldots,y_{n}^{\pm1}\right]^{S_{m}\times S_{n}}$
by the ideal\emph{ $\left\langle x_{1}\ldots x_{m}-y_{1}\ldots y_{n}\right\rangle $.}
Clearly, $f\vert_{\beta=0}=f\vert_{x_{i}=y_{j}}=0$ for any $f\in J\left(\mathfrak{g}\right)_{a}$.
Hence, $ds_{x}\left(f\right)=0$ for any $x\in X_{\mathfrak{g}}$
and $f\in J\left(\mathfrak{g}\right)_{a}$. 

If $\mathfrak{\mathfrak{g}}=B\left(m|n\right)$, $C\left(n+1\right)$
or $D\left(m|n\right)$, then $\mathcal{J}_{\mathfrak{g}}=\mathcal{J}_{G}\oplus\tilde{{\mathcal{J}}}$
and $ds_{x}\left(f\right)=0$ for all $f\in\tilde{{\mathcal{J}}}$.
Indeed, for $\beta=\pm\varepsilon_{i}\pm\delta_{j}$ it is not difficult
to check that $f\vert_{\beta=0}=f\vert_{x_{i}^{\pm1}=y_{j}^{\pm1}}=f\vert_{u_{i}=v_{j}}=0$.

The supercharacter ring of $\mathfrak{\mathfrak{g}}=B\left(m|n\right)$
is $\mathcal{J}_{\mathfrak{g}}=\mathcal{J}_{G}\oplus J_{\mathfrak{g,}\frac{1}{2}}$,
where 
\[
J_{\mathfrak{g},\frac{1}{2}}=\left\{ \prod_{i=1}^{m}\left(x_{i}^{\frac{1}{2}}+x_{i}^{-\frac{1}{2}}\right)\prod_{i,j}\left(u_{i}-v_{j}\right)g\mid g\in\mathbb{Z}\left[u_{1},\ldots,u_{m},v_{1},\ldots,v_{n}\right]^{S_{m}\times S_{n}}\right\} .
\]

The supercharacter ring of $\mathfrak{g}=C\left(n+1\right)$ is $\mathcal{J}_{\mathfrak{g}}=\mathcal{J}_{G}\oplus\left(J\left(\mathfrak{g}\right)_{0}^{-}\oplus\bigoplus_{a\in\mathbb{C}/\mathbb{Z}}J\left(\mathfrak{g}\right)_{a}\right)$,
where
\[
J\left(\mathfrak{g}\right)_{0}^{-}=\left\{ x_{1}\prod_{j=1}^{n}\left(u_{1}-v_{j}\right)g\mid g\in\mathbb{Z}\left[u_{1},v_{1},\ldots,v_{n}\right]^{S_{n}}\right\} 
\]
\[
J\left(\mathfrak{g}\right)_{a}=x_{1}^{a}\prod_{j=1}^{n}\left(1-x_{1}y_{j}\right)\left(1-x_{1}y_{j}^{-1}\right)\mathbb{Z}\left[x_{1}^{\pm1},y_{1}^{\pm1},\ldots,y_{n}^{\pm1}\right]^{W},
\]

The supercharacter ring of $\mathfrak{\mathfrak{g}}=D\left(m|n\right)$
is $\mathcal{J}_{\mathfrak{g}}=\mathcal{J}_{G}\oplus\left(J\left(\mathfrak{g}\right)_{0}^{-}\oplus J_{\mathfrak{g,}\frac{1}{2}}\right)$,
where 
\[
J\left(\mathfrak{g}\right)_{0}^{-}=\left\{ \omega\prod_{i,j}\left(u_{i}-v_{j}\right)g\mid g\in\mathbb{Z}\left[u_{1},\ldots,u_{m},v_{1},\ldots,v_{n}\right]^{S_{m}\times S_{n}}\right\} 
\]
\[
J_{\mathfrak{g},\frac{1}{2}}=\prod_{i,j}\left(u_{i}-v_{j}\right)\left(\left(x_{1}\ldots x_{m}\right)^{\frac{1}{2}}\mathbb{Z}\left[u_{1},\ldots,u_{m},v_{1},\ldots,v_{n}\right]\right)^{W}.
\]
\end{proof}
\begin{prop}
Let $\mathfrak{\mathfrak{g}=gl}\left(m|n\right)$\textup{\emph{ and
$x\in X_{\mathfrak{g}}$. }}The image of $ds_{x}:\mathcal{J}_{\mathfrak{g}}\rightarrow\mbox{\ensuremath{\mathcal{J}}}_{\mathfrak{g}_{x}}$
is 
\[
\bigoplus_{a\in\mathbb{C}/\mathbb{Z}}\left(x_{1}\cdots x_{m-k}\right)^{a}\left(y_{1}\cdots y_{n-k}\right)^{-a}\mathcal{J}_{G_{x}},
\]
 where $k$ is the size of $\psi\left(x\right)\in S_{\mathfrak{g}}$
under the bijection $\psi:X_{\mathfrak{g}}/G_{\bar{0}}\rightarrow S_{\mathfrak{g}}/W$,
and $\mathcal{J}_{G_{x}}$ is the supercharacter ring of the Lie supergroup
$G_{x}$.
\end{prop}
\begin{proof}
By Sergeev and Veselov in \cite{SV}, the supercharacter ring of $\mathfrak{gl}\left(m|n\right)$
is $\mathcal{J}_{\mathfrak{g}}=\bigoplus_{a,b\in\mathbb{C}/\mathbb{Z}}J\left(\mathfrak{g}\right)_{a,b}$
where $J\left(\mathfrak{g}\right)_{0,0}=\mathcal{J}_{G}$ , 
\[
J\left(\mathfrak{g}\right)_{a,b}=\left(x_{1}\cdots x_{m}\right)^{a}\left(y_{1}\cdots y_{n}\right)^{-a}J\left(\mathfrak{g}\right)_{0,0}
\]
 when $a+b\in\mathbb{Z}$, but $a\not\in\mathbb{Z}$, and

\[
J\left(\mathfrak{g}\right)_{a,b}=\left(x_{1}\cdots x_{m}\right)^{a}\left(y_{1}\cdots y_{n}\right)^{b}\Pi_{i,j}\left(1-x_{i}y_{j}^{-1}\right)\mathbb{Z}\left[x_{1}^{\pm1},\ldots,x_{m}^{\pm1},y_{1}^{\pm1},\ldots,y_{n}^{\pm1}\right]^{S_{m}\times S_{n}}
\]
when $a+b\not\in\mathbb{Z}$. 

Then we have that $f\vert_{x_{i}=y_{j}}=0$ for any $f\in J\left(\mathfrak{g}\right)_{a,b}$
with $a+b\not\in\mathbb{Z}$. By Theorem \ref{thm:surjectivity},
$ds_{x}\left(J\left(\mathfrak{g}\right)_{0,0}\right)=J\left(\mathfrak{g}_{x}\right)_{0,0}$.
Since $ds_{x}\left(f\right)=f\vert_{x_{r_{i}}=y_{s_{i}};i=1,\ldots,k}$
by Theorem \ref{prop:realization dt}, we have that 
\[
ds_{x}\left(J\left(\mathfrak{g}\right)_{a,b}\right)=\left(x_{1}\cdots x_{m-k}\right)^{a}\left(y_{1}\cdots y_{n-k}\right)^{-a}J\left(\mathfrak{g}_{x}\right)_{0,0}=J\left(\mathfrak{g}_{x}\right)_{a,b}
\]
when $a+b\in\mathbb{Z}$, but $a\not\in\mathbb{Z}$.
\end{proof}

\subsection{The image of $ds_{x}$ for the exceptional Lie superalgebras. }

In this section, we describe the image of $ds_{x}$ for the Lie superalgebras
$G\left(3\right)$, $F\left(4\right)$ and $D\left(2,1,\alpha\right)$,
using the explicit description of the supercharacter rings given by
Sergeev and Veselov in \cite[Section 7]{SV}.

Since $G\left(3\right)$, $F\left(4\right)$ and $D\left(2,1,\alpha\right)$
have defect 1, we may assume that $x\in\mathfrak{g}_{\beta}$ for
some isotropic root $\beta$. Moreover, since $W\beta=\Delta_{\bar{1}}$,
it suffices to describe the image for a fixed choice of $\beta$.

\subsubsection{$G\left(3\right)$}

Let $\beta=\varepsilon_{3}+\delta_{1}$. Then $\mathfrak{g}_{x}\cong\mathfrak{sl}\left(2\right)$
with $\Delta_{x}=\left\{ \pm\left(\varepsilon_{1}-\varepsilon_{2}\right)\right\} $.
The supercharacter ring of $G\left(3\right)$ is 
\[
\mathcal{J}_{\mathfrak{g}}=\left\{ g\left(w\right)+\left(v_{1}-u_{1}\right)\left(v_{1}-u_{2}\right)\left(v_{1}-u_{3}\right)h\mid h\in\mathbb{Z}\left[u_{1},u_{2},u_{3},v_{1}\right]^{S_{3}},\ g\in\mathbb{Z}\left[w\right]\right\} ,
\]
where $y_{1}=e^{\delta_{1}}$, $v_{1}=y_{1}+y_{1}^{-1}$, $x_{i}=e^{\varepsilon_{i}}$,
$u_{i}=x_{i}+x_{i}^{-1}$ for $i=1,2,3$ and
\[
w=v_{1}^{2}-v_{1}\left(u_{1}+u_{2}+u_{3}+1\right)+u_{1}u_{2}+u_{1}u_{3}+u_{2}u_{3}.
\]
 Note that $x_{1}x_{2}x_{3}=1$, so $u_{3}=x_{1}x_{2}+x_{1}^{-1}x_{2}^{-1}$. 

Theorem \ref{prop:realization dt} implies that $ds_{x}\left(f\right)=f\mid_{y_{1}=x_{3}^{-1}=x_{1}x_{2}}$
for every $f\in\mathcal{J}_{\mathfrak{g}}$. Hence, $ds_{x}\left(f\right)=ds_{x}\left(g\left(w\right)\right)$
since $\left(v_{1}-u_{3}\right)\mid_{y_{1}=x_{3}^{-1}=x_{1}x_{2}}=0$.
Thus, the image of $ds_{x}$ is the polynomial ring $\mathbb{Z}\left[w_{x}\right]$
generated by the element

\[
w_{x}:=w\mid_{y_{1}=x_{3}^{-1}=x_{1}x_{2}}=\frac{x_{1}}{x_{2}}+\frac{x_{2}}{x_{1}}\in\mathcal{J}_{\mathfrak{g}_{x}}.
\]
Note that $w_{x}+1$ is the supercharacter of the adjoint representation
of $\mathfrak{sl}\left(2\right)$, and that $\frac{x_{1}}{x_{2}}+\frac{x_{2}}{x_{1}}$
equals $x_{1}^{2}+x_{2}^{2}$ in $\mathcal{J}_{\mathfrak{g}_{x}}$
due to the relation $x_{1}x_{2}=1$. Finally, we obtain that
\[
\mathrm{Im}\ ds_{x}=\mathbb{Z}\left[x_{1}^{2}+x_{2}^{-2}\right]\subsetneq\mathcal{J}_{G_{x}}=\mathcal{J}_{SL\left(2\right)}=\mathbb{Z}\left[x_{1}^{\pm1},x_{2}^{\pm1}\right]^{S_{2}}/\left\langle x_{1}x_{2}-1\right\rangle \cong\mathbb{Z}\left[x_{1}+x_{1}^{-1}\right].
\]

\subsubsection{$F\left(4\right)$}

Let $\beta=\frac{1}{2}\left(\varepsilon_{1}+\varepsilon_{2}+\varepsilon_{3}-\delta_{1}\right)$.
Then $\mathfrak{g}_{x}\cong\mathfrak{sl}\left(3\right)$ with $\Delta_{x}=\left\{ \varepsilon_{i}-\varepsilon_{j}\mid1\leq i,j\leq3\right\} $.
The supercharacter ring of $F\left(4\right)$ is
\[
\mathcal{J}_{\mathfrak{g}}=\left\{ g\left(w_{1},w_{2}\right)+Qh\mid h\in\mathbb{Z}\left[x_{1}^{\pm2},x_{2}^{\pm2},x_{3}^{\pm2},\left(x_{1}x_{2}x_{3}\right)^{\pm1},y_{1}^{\pm1}\right]^{W_{0}},\ g\in\mathbb{Z}\left[w_{1},w_{2}\right]\right\} ,
\]
where $y_{1}=e^{\frac{1}{2}\delta_{1}}$, $x_{i}=e^{\frac{1}{2}\varepsilon_{i}}$
for $i=1,2,3$, and
\begin{align*}
Q & =\left(y_{1}+y_{1}^{-1}-x_{1}x_{2}x_{3}-x_{1}^{-1}x_{2}^{-1}x_{3}^{-1}\right)\prod_{i=1}^{3}\left(y_{1}+y_{1}^{-1}-\frac{x_{1}x_{2}x_{3}}{x_{i}^{2}}-\frac{x_{i}^{2}}{x_{1}x_{2}x_{3}}\right)\\
w_{k} & =\sum_{i\ne j}\frac{x_{i}^{2k}}{x_{j}^{2k}}+\sum_{i=1}^{3}\left(x_{i}^{2k}+x_{i}^{-2k}\right)+y_{1}^{2k}+y_{1}^{-2k}-\left(y_{1}^{k}+y_{1}^{-k}\right)\prod_{i=1}^{3}\left(x_{i}^{k}+x_{i}^{-k}\right),\quad k=1,2.
\end{align*}
Theorem \ref{prop:realization dt} implies that $ds_{x}\left(f\right)=f\mid_{x_{1}x_{2}x_{3}=y_{1}}$
for every $f\in\mathcal{J}_{\mathfrak{g}}$. Hence, $ds_{x}\left(f\right)=ds_{x}\left(g(w_{1},w_{2})\right)$
since $Q\mid_{x_{1}x_{2}x_{3}=y_{1}}=0$. Thus, the image of $ds_{x}$
is generated by the elements

\begin{align*}
w_{x}^{1}:=w_{1}\mid_{x_{1}x_{2}x_{3}=y_{1}} & =\sum_{i\ne j}\frac{x_{i}^{2}}{x_{j}^{2}}\\
w_{x}^{2}:=w_{2}\mid_{x_{1}x_{2}x_{3}=y_{1}} & =\sum_{i\ne j}\frac{x_{i}^{4}}{x_{j}^{4}},
\end{align*}
and is a proper subring of $\mathcal{J}_{G_{x}}=\mathcal{J}_{SL(3)}=\mathbb{Z}\left[x_{1}^{\pm1},x_{2}^{\pm1},x_{3}^{\pm}\right]^{S_{3}}/\left\langle x_{1}x_{2}x_{3}-1\right\rangle $.

\subsubsection{$D\left(2,1,\alpha\right)$}

Let $\beta=\varepsilon_{1}-\varepsilon_{2}-\varepsilon_{3}$. Then
$\mathfrak{g}_{x}\cong\mathbb{C}$. 

If $\alpha\not\in\mathbb{Q}$, then the supercharacter ring of $D\left(2,1,\alpha\right)$
is
\[
\mathcal{J}_{\mathfrak{g}}=\left\{ c+Qh\mid c\in\mathbb{Z},\ h\in\mathbb{Z}\left[u_{1},u_{2},u_{3}\right]\right\} ,
\]
where $x_{i}:=e^{\varepsilon_{i}}$, $u_{i}=x_{i}+x_{i}^{-1}$ for
$i=1,2,3$ and
\begin{align*}
Q & =\left(x_{1}-x_{2}x_{3}\right)\left(x_{2}-x_{1}x_{3}\right)\left(x_{3}-x_{1}x_{2}\right)\left(1-x_{1}x_{2}x_{3}\right)x_{1}^{-2}x_{2}^{-2}x_{3}^{-2}\\
 & =u_{1}^{2}+u_{2}^{2}+u_{3}^{2}-u_{1}u_{2}u_{3}-4.
\end{align*}

If $\alpha=\frac{p}{q}\in\mathbb{Q}$, then the supercharacter ring
of $D\left(2,1,\alpha\right)$ is
\[
\mathcal{J}_{\mathfrak{g}}=\left\{ g\left(w_{\alpha}\right)+Qh\mid g\in\mathbb{Z}\left[w_{\alpha}\right],\ h\in\mathbb{Z}\left[u_{1},u_{2},u_{3}\right]\right\} ,
\]
where
\[
w_{\alpha}=\left(x_{1}+x_{1}^{-1}-x_{2}x_{3}-x_{2}^{-1}x_{3}^{-1}\right)\frac{\left(x_{2}^{p}-x_{2}^{-p}\right)\left(x_{3}^{q}-x_{3}^{-q}\right)}{\left(x_{2}-x_{2}^{-1}\right)\left(x_{3}-x_{3}^{-1}\right)}+x_{2}^{p}x_{3}^{-q}+x_{2}^{-p}x_{3}^{q}.
\]
By Theorem \ref{prop:realization dt}, $ds_{x}\left(f\right)=f\mid_{x_{1}=x_{2}x_{3}}$
for every $f\in\mathcal{J}_{\mathfrak{g}}$. Since $Q\mid_{x_{1}=x_{2}x_{3}}=0$,
$ds_{x}\left(f\right)=c$ for some $c\in\mathbb{Z}$ when $\alpha\not\in\mathbb{Q}$,
while $ds_{x}\left(f\right)=ds_{x}\left(g\left(w_{\alpha}\right)\right)$
when $\alpha\in\mathbb{Q}$. Thus the image of $ds_{x}$ is $\mathbb{\mathbb{Z}}\subset\mathcal{J}_{\mathbb{C}}$
when $\alpha\not\in\mathbb{Q}$ and the image is $\mathbb{\mathbb{Z}}\left[w_{\alpha}\right]\subset\mathcal{J}_{\mathbb{C}}$
when $\alpha\in\mathbb{Q}$.

C.H.: Dept. of Mathematics, Weizmann Institute \& ORT Braude College,
crystal.hoyt@weizmann.ac.il\\
 S.R.: Dept. of Mathematics, Bar-Ilan University, shifra.reif@biu.ac.il
\end{document}